\documentclass[submission]{FPSAC2025}

\usepackage{mathtools,bm,ytableau,multicol}
\newtheorem{thm}{Theorem}

\newtheorem{prop}{Proposition}
\newtheorem{cor}{Corollary}
\newtheorem{conj}{Conjecture}

\theoremstyle{definition}
\newtheorem{defn}{Definition}
\newtheorem{example}{Example}

\title[]{Monotonicity for generalized binomial coefficients and Jack positivity}

\author[]{Hong Chen \thanks{\href{mailto:hc813@math.rutgers.edu}{hc813@math.rutgers.edu}}\addressmark{1} \and Siddhartha Sahi\thanks{\href{mailto:sahi@math.rutgers.edu}{sahi@math.rutgers.edu}}\addressmark{1}}

\address{\addressmark{1}Department of Mathematics, Rutgers University, NJ, USA}

\received{\today}

\abstract{%
	Binomial formulas for Schur polynomials and Jack polynomials were studied by Lascoux in 1978, and Kaneko, Okounkov--Olshanski and Lassalle in the 1990s.
	We prove that the associated binomial coefficients are monotone and derive some symmetric function inequalities, in particular, a Schur positivity and Jack positivity result.
	These inequalities are similar to those studied by Newton, Muirhead, Gantmacher, Cuttler--Greene--Skandera, Sra and Khare--Tao.
}


\usepackage[backend=bibtex]{biblatex}
\def\({\left(}
\def\){\right)}
\def\<{\langle}
\def\>{\rangle}

\let\=\relax
\newcommand{\=}{\mathrel{\phantom{=}}}

\def\Z{\mathbb Z}
\def\Q{\mathbb Q}
\def\R{\mathbb R}
\def\F{\mathbb F}
\def\fp{\mathbb F_{\geqslant0}}
\def\fpp{\mathbb F_{>0}}

\def\AJ{{A\mathrm{J}}}

\def\monic{{\mathrm{monic}}}

\def\geq{\geqslant}
\def\leq{\leqslant}
\newcommand{\wm}{\mathrel{\succcurlyeq_\mathrm{w}}}
\newcommand{\m}{\mathrel{\succcurlyeq}}

\newcommand*{\SetSuchThat}[1][]{} 

\newcommand*{\MvertSets}{%
	\renewcommand*\SetSuchThat[1][]{%
		\mathclose{}%
		\nonscript\;##1\vert\penalty\relpenalty\nonscript\;%
		\mathopen{}%
	}%
}
\MvertSets 
\DeclarePairedDelimiterX \Set [2] {\lbrace}{\rbrace}
{\,#1\SetSuchThat[\delimsize]#2\,}

\newcommand{\mydef}[1]{{\color{blue}\textit{#1}}}

\begin{document}

\maketitle

\section{Introduction}\label{sec:intro}
Schur polynomials are ubiquitous in math, in particular, in algebra, representation theory and combinatorics.  
They form an orthonormal basis of the ring of symmetric polynomials, correspond to the irreducible (polynomial) representations of the symmetric group and the general linear group, and are the generating functions of semi-standard Young tableaux \cite{Sagan91,Stanley,Mac15,}.

A symmetric polynomial is called \mydef{Schur positive}, if its expansion coefficients in terms of Schur polynomials are all non-negative.
Schur positivity has long been a curious and difficult question.
The first Schur positivity result is that products of Schur polynomials are Schur positive via the celebrated \mydef{Littlewood--Richardson rule}:
\begin{align*}
	s_\mu s_\nu = \sum_\lambda c_{\mu\nu}^\lambda s_\lambda,\quad  c_{\mu\nu}^\lambda\geq0.
\end{align*}
There are many proofs and combinatorial interpretations of this rule, e.g., \cite{BZ92,KT99,Stanley,Mac15}.

\mydef{Jack polynomials} are a one-parameter deformation of Schur polynomials, defined over the base field $\F=\Q(\tau)$.
The Jack polynomial $P_\lambda(x;\tau)$ reduces to the Schur polynomial $s_\lambda(x)$ when $\tau=1$ (and to many other bases for other $\tau$).
Jack polynomials originated from multivariate statistics \cite{Jack,Muirhead82}, but later turned out to have rich combinatorial structures \cite{St89,KSinv,Mac15,DD23} and play an important role in the study of Selberg integrals \cite{Kan93,Kad97}, the Calogero--Sutherland model \cite{LV96} and random partitions \cite{Oko05}.

\subsection{Symmetric function inequalities}
In this extended abstract, we prove the following Schur positivity and Jack positivity result. 
For the full details and proofs, see \cite{CS24}.
\begin{thm}\label{thm:1}
	Let $\lambda$ and $\mu$ be partitions of length at most $n$, $x=(x_1,\dots,x_n)$ and $\bm1=(1^n)$. Then the following are equivalent:
	\begin{enumerate}
		\item $\lambda\supseteq\mu$, i.e., $\lambda_i\geq\mu_i$, $1\leq i\leq n$.
		\item $\displaystyle \frac{s_\lambda(x+\bm1)}{s_\lambda(\bm1)}-\frac{s_\mu(x+\bm1)}{s_\mu(\bm1)}$ is \textbf{Schur positive}.
		\item $\displaystyle \frac{P_\lambda(x+\bm1;\tau)}{P_\lambda(\bm1;\tau)}-\frac{P_\mu(x+\bm1;\tau)}{P_\mu(\bm1;\tau)}$ is \textbf{Jack positive}.
	\end{enumerate}
	Here, the Jack positivity is over the cone  $\fp=\Set{f/g\in\Q[\tau]}{f,g\in\Z_{\geq0}[\tau],g\neq0}\subset\F$.
\end{thm}

\begin{example}
	Write $S(\lambda)(x)=s_\lambda(x)/s_\lambda(\bm1)$ and $\widetilde S(\lambda)(x)=S(\lambda)(x+\bm1)$ for Schur polynomials and similarly $P^*(\lambda)$ and $\widetilde P^*(\lambda)$ for Jack polynomials, then we have 
	\newcommand{\mysizeone}{0.4}
	\begin{align*}
		\widetilde S({\scalebox{\mysizeone}{\ydiagram{3,1}}})
		&&-&&\widetilde S({\scalebox{\mysizeone}{\ydiagram{2}}})
		&&=&&S({\scalebox{\mysizeone}{\ydiagram{3,1}}})
		&&+&&\tfrac43 S({\scalebox{\mysizeone}{\ydiagram{3}}})
		&&+&&\tfrac83 S({\scalebox{\mysizeone}{\ydiagram{2,1}}})
		\\&&&&&&+&&3S({\scalebox{\mysizeone}{\ydiagram{2}}})
		&&+&&2S({\scalebox{\mysizeone}{\ydiagram{1,1}}})
		&&+&&2S({\scalebox{\mysizeone}{\ydiagram{1}}})
		\;;
		\\
		\widetilde P^*({\scalebox{\mysizeone}{\ydiagram{3,1}}})
		&&-&&\widetilde P^*({\scalebox{\mysizeone}{\ydiagram{2}}})
		&&=&&P^*({\scalebox{\mysizeone}{\ydiagram{3,1}}})
		&&+&&\tfrac{2\tau+2}{\tau+2}P^*({\scalebox{\mysizeone}{\ydiagram{3}}})
		&&+&&\tfrac{2\tau+6}{\tau+2}P^*({\scalebox{\mysizeone}{\ydiagram{2,1}}})
		\\&&&&&&+&&\tfrac{4\tau+2}{\tau+1}P^*({\scalebox{\mysizeone}{\ydiagram{2}}})
		&&+&&\tfrac{\tau+3}{\tau+1}P^*({\scalebox{\mysizeone}{\ydiagram{1,1}}})
		&&+&&2P^*({\scalebox{\mysizeone}{\ydiagram{1}}})
		\;.
	\end{align*}
	All coefficients are in $\Q_{\geq0}$ or $\fp$. 
\end{example}
Note that when $\tau=1$, the coefficients for Jack polynomials specialize to those for Schur polynomials.
By specializing $\tau$ to other values, one can get monomial positivity, elementary positivity and Zonal positivity results, see \cref{thm:CS-ineq}.

Our Schur positivity (and monomial positivity and elementary positivity) result is closely related to the inequalities studied by Cuttler--Greene--Skandera, Sra and Khare--Tao in \cite{CGS11,Sra16,KT21}, which we now recall.
\begin{prop}\label{prop:1}
	Let $\lambda$, $\mu$, $x$, and $\bm1$ be as in \cref{thm:1}.
	\begin{enumerate}
		\item (CGS--Sra) Suppose $|\lambda|=|\mu|$. Then $\lambda$ dominates $\mu$ if and only if 
		\begin{align*}\label{eqn:CGS-Schur}
			\dfrac{s_\lambda(x)}{s_\lambda(\bm1)}-\dfrac{s_\mu(x)}{s_\mu(\bm1)}\geq0,\quad x\in[0,\infty)^n.
		\end{align*}
		\item (KT) $\lambda$ weakly dominates $\mu$ if and only if 
		\begin{align*}
			\dfrac{s_\lambda(x+\bm1)}{s_\lambda(\bm1)}-\dfrac{s_\mu(x+\bm1)}{s_\mu(\bm1)}\geq0,\quad x\in[0,\infty)^n.
		\end{align*}
	\end{enumerate}
	Here $\lambda$ \mydef{weakly dominates} $\mu$ means that $\lambda_1+\cdots+\lambda_i\geq\mu_1+\cdots+\mu_i$ for $1\leq i\leq n$; and $\lambda$ \mydef{dominates} $\mu$ means, in addition, that $|\lambda|=|\mu|$.
\end{prop}
In \cite{CGS11}, the authors studied inequalities of means for various symmetric polynomials.
The first inequality above, in the case of the monomial symmetric function, is the classical Muirhead's inequality; in the case of elementary symmetric function and power-sum, generalizes some inequalities due to Newton and Gantmacher, respectively.
They made a conjecture for Schur polynomials, which was later proved by Sra in \cite{Sra16} using the HCIZ integral and the AM--GM inequality.

In \cite{KT21}, the authors studied entrywise-functions that preserve the positive semidefiniteness of matrices and provided a complete characterization of the sign patterns of the higher-order Maclaurin coefficients of such functions. Using a Schur positivity result of Lam--Postnikov--Pylyavskyy, they turned from a qualitative existence result to quantitative bounds on the coefficients. As an application, they gave a new characterization of weak majorization in terms of Schur polynomials (the second part of \cref{prop:1}).

Each of statements in \cref{thm:1,prop:1} is a characterization of certain partial order on partitions.
The two inequalities in \cref{prop:1} are \mydef{evaluation positivity}, while in \cref{thm:1}, we have \mydef{expansion positivity}.
Since Schur polynomials are evaluation positive, it follows that the difference in \cref{thm:1} is also positive when evaluating at the positive orthant $[0,\infty)^n$.
In fact, \cref{thm:1}, together with the first inequality in \cref{prop:1}, implies the second one, see \cref{sec:ineq}.

\subsection{Binomial formulas}
The numerator $s_\lambda(x+\bm1)$ in \cref{thm:1} was first considered by Lascoux in \cite{Lascoux} in order to compute the Chern classes of the exterior and symmetric squares of a vector bundle, see also \cite[P.47~Example~10]{Mac15}.
This is a natural generalization of the well-known Newton's binomial formula to symmetric polynomials.

Later, in the 1990s, binomial formulas for Schur polynomials, Jack polynomials, Macdonald polynomials (and their non-symmetric counterparts) and Koornwinder polynomials were studied by Lassalle, Kaneko, Sahi, Okounkov--Olshanski and others, in \cite{Las90,Kan93,OO-schur,OO97,Oko98-Mac,Sahi98}. See also \cite{Koo15} for a good survey.

It was shown in \cite{OO97} for Jack polynomials that 
\begin{align*}
	\frac{P_\lambda(x+\bm1;\tau)}{P_\lambda(\bm1;\tau)}
	= \sum_\nu \binom{\lambda}{\nu}_\tau \frac{P_\nu(x;\tau)}{P_\nu(\bm1;\tau)},
\end{align*}
where the \mydef{generalized binomial coefficients} $\binom{\lambda}{\nu}_\tau$ are given by evaluations of the \mydef{interpolation Jack polynomials}, first studied in \cite{Sahi94,KS96}.
Later in \cite{Sahi-Jack}, it was shown that $\binom{\lambda}{\nu}_\tau$ is always positive (that is, in $\fp$).
In this paper, we prove that $\binom{\lambda}{\nu}_\tau$ is increasing in $\lambda$, which implies \cref{thm:1}. 
See \cref{sec:mono} for a sketch of the proof.
\begin{thm}\label{thm:2}
	If $\lambda\supseteq\mu$, then $\binom{\lambda}{\nu}_\tau-\binom{\mu}{\nu}_\tau\in\fp$ for all $\nu$.
\end{thm}

\section{Background}\label{sec:background}
Throughout the paper, let $n\geq1$ be the number of variables and $x=(x_1,\dots,x_n)$.
Recall the following notions in \cite{Sagan91,Stanley,Mac15}.
\subsection{Partitions}
A \mydef{partition} of length at most $n$ is a tuple $\lambda=(\lambda_1,\dots,\lambda_n)\in\Z^n$ such that $\lambda_1\geq\lambda_2\geq\cdots\geq\lambda_n\geq0.$
The \mydef{size} of $\lambda$ is $|\lambda|=\lambda_1+\lambda_2+\dots+\lambda_n$, and the \mydef{length} $\ell(\lambda)$ is the number of nonzero components.
Let $m_i(\lambda)$ be the number of components equal to $i$ in $\lambda$.
Let $\mathcal P_n$ be the set of all partitions of length at most $n$.

For two partitions $\lambda$ and $\mu$, consider the following partial orders:

We say $\lambda$ \mydef{contains} $\mu$ and write $\lambda\supseteq\mu$ if $\lambda_i\geq\mu_i$ for $1\leq i\leq n$. 

We say $\lambda$ \mydef{weakly dominates} or \mydef{weakly majorizes} $\mu$ and write $\lambda\wm\mu$ if $\sum_{i=1}^k\lambda_i\geq\sum_{i=1}^k\mu_i$ for $1\leq k\leq n$. 

We say $\lambda$ \mydef{dominates} or \mydef{majorizes} $\mu$ and write $\lambda\geq\mu$ if $|\lambda|=|\mu|$ and $\lambda\wm\mu$.

\subsection{Diagrams and tableaux}
For a partition $\lambda$, we usually identify it with its \mydef{Young diagram} (or Ferrers diagram) , which is a set of lattice points (usually drawn as boxes),
$	\lambda=\Set{(i,j)\in\Z^2}{1\leq j\leq \lambda_i,1\leq i\leq l(\lambda)}.$
A \mydef{tableau} $T$ of shape $\lambda$ and rank $N$ is a function $T:\lambda\to\{1,2,\dots,N\}$. 
A \mydef{semi-standard reverse tableau} (RT, for short) is one such that $T(i,j)>T(i+1,j)$ and $T(i,j)\geq T(i,j+1)$ for $(i,j),(i+1,j),(i,j+1)\in\lambda$.

For example, for the partition $(3,2)$, its diagram and all RTs of rank $2$ are as follows:
\newcommand{\mysizetwo}{0.5}
\begin{align*}
	\scalebox{\mysizetwo}{\ydiagram{3,2}}\ ,\quad \scalebox{\mysizetwo}{\ytableaushort{221,11}}\ ,\quad \scalebox{\mysizetwo}{\ytableaushort{222,11}}\ .
\end{align*}

The \mydef{conjugate} of a partition $\lambda$ is the partition $\lambda'$ associated to the transpose of the diagram of $\lambda$.
In the example above, $(3,2)'=(2,2,1)$.

\subsection{Symmetric polynomials}
Denote by $\Lambda=\Lambda_n\coloneqq\Z[x_1,\dots,x_n]^{S_n}$ the ring of symmetric polynomials over $\Z$ in $n$ variables and $\Lambda_\F=\Lambda\otimes_\Z\F$.

The \mydef{monomial symmetric polynomial} $m_\lambda$ is defined as
$m_\lambda = \sum_{\alpha\sim\lambda} x^\alpha$,
where the sum runs over permutations of $\lambda$ and $x^\alpha = x_1^{\alpha_1}x_2^{\alpha_2}\cdots x_n^{\alpha_n}$.

For each $r\geq1$, the \mydef{elementary symmetric polynomial} $e_r$ and the \mydef{power-sum} $p_r$ are 
\begin{align}
	e_r = m_{(r)}=\sum_{1\leq i_1<\cdots<i_r\leq n} x_{i_1}\cdots x_{i_r},\quad p_r = m_{(1^r)} = \sum_{1\leq i\leq n} x_i^r.
\end{align}
Define multiplicatively, $e_\lambda = e_{\lambda_1}\cdots e_{\lambda_l}$ and $p_\lambda=p_{\lambda_1}\cdots p_{\lambda_l}$, where $l=\ell(\lambda)$.

The \mydef{Schur polynomial} $s_\lambda$ can be defined by the bialternant formula
\begin{align}
	s_\lambda = \frac{\det(x_i^{\lambda_j+n-j})_{1\leq i,j\leq n}}{\det(x_i^{n-j})_{1\leq i,j\leq n}}.
\end{align}

The \mydef{Hall inner product} on $\Lambda_\Q$ is defined by 
\begin{align}
	\<p_\lambda,p_\mu\> = \delta_{\lambda\mu} z_\lambda,
\end{align}
where $z_\lambda = \prod_{i\geq1} i^{m_i(\lambda)} m_i(\lambda)!$ and $\delta_{\lambda\mu}$ is the Kronecker delta function.

It is a well-known fact that Schur polynomials form an \mydef{orthonormal} basis of $\Lambda_n$ that is \mydef{unitriangular} under the monomial basis with respect to the dominance order:
\begin{gather}
	\<s_\lambda,s_\mu\> = \delta_{\lambda\mu},	\quad
	s_\lambda = m_\lambda + \sum_{\mu<\lambda} K_{\lambda\mu} m_\mu.
\end{gather}

\mydef{Jack polynomials} are deformations of Schur polynomials.
The base field is $\F=\Q(\tau)$, the field of rational polynomials in the parameter $\tau$. (Note: our $\tau$ corresponds to $1/\alpha$ for the parameter $\alpha$ in \cite[Section VI.10]{Mac15}.)
Define an inner product on $\Lambda_\F$ by 
\begin{gather}
	\<p_\lambda,p_\mu\>_\tau = \delta_{\lambda\mu} z_\lambda \tau^{-l(\lambda)}.
\end{gather}
Then, similar to Schur polynomials, Jack polynomials are uniquely determined by the following \mydef{orthogonality} and \mydef{unitriangularity} conditions:
\begin{gather}
	\<P_\lambda(x;\tau),P_\mu(x;\tau)\>_\tau = 0, \quad \lambda\neq\mu;	\quad
	P_\lambda = m_\lambda + \sum_{\mu<\lambda} u_{\lambda\mu} m_\mu.
\end{gather}
The normalization here is sometimes called \mydef{monic}.

Jack polynomials specialize to many families.
For example, when $\tau=0$, $P_\lambda=m_\lambda$; when $\tau=1$, $P_\lambda=s_\lambda$; when $\tau=\infty$, $P_\lambda=e_{\lambda'}$.
Also, when $\tau=1/2$ or $2$, $P_\lambda$ specializes to the Zonal polynomial $Z_\lambda$, see \cite[Chapter VII]{Mac15}.
\subsection{Binomial formulas and interpolation polynomials}
The classical Newton's binomial formula states that in the \emph{univariate} case
\begin{align}
	(x+1)^k = \sum_{m\geq0} \binom{k}{m} x^m,
\end{align}
where the binomial coefficient $\binom{k}{m}=\frac{k(k-1)\cdots(k-m+1)}{m!}$ is a polynomial in $k$.

In \cite{OO97}, the following binomial formula for Jack polynomials was proved:
\begin{align}\label{eqn:AJ-binomial1}
	\frac{P_\lambda(x+\bm1;\tau)}{P_\lambda(\bm1;\tau)} = \sum_{\mu} \binom{\lambda}{\nu}_\tau \frac{P_\nu(x;\tau)}{P_\nu(\bm1;\tau)},
\end{align}
where $\bm1=(1^n)=(1,1,\dots,1)$.
The coefficients $\binom{\lambda}{\nu}_\tau$ are called the \mydef{generalized binomial coefficients}. 
They, too, are given by evaluations of certain polynomials, called \mydef{interpolation Jack polynomials} (or shifted Jack polynomials), first studied in \cite{Sahi94,KS96}.

To define the interpolation Jack polynomials, we need the following result about interpolation in \cite{KS96}.
Let $\overline\lambda_i=\lambda_i+(n-i)\tau$.
\begin{prop}\label{prop:interpolation}
	Fix $d\geqslant0$ and any function $\overline f:\Set{\lambda\in\mathcal P_n}{|\lambda|\leq d}\to\mathbb F$. 
	Then there is a unique symmetric polynomial $f\in\Lambda_\F$, of degree at most $d$, such that $f(\overline \lambda)=\overline f(\lambda)$ for $|\lambda|\leq d$.
\end{prop}
\begin{defn}
	The interpolation Jack polynomial $h_\mu(x;\tau)$ is the unique symmetric polynomial of degree $|\mu|$ that interpolates the Kronecker delta function (for $|\lambda|\leq|\mu|$), i.e., 
	\begin{gather}
		h_\mu(\overline\lambda;\tau) = \delta_{\lambda\mu}, \quad\forall \lambda\in\mathcal P_n,\ |\lambda|\leqslant|\mu|, \label{eqn:def-vanishing}\\
		\deg h_\mu =|\mu|\label{eqn:def-deg}.
	\end{gather}
	The generalized binomial coefficients above are given by evaluations:
	\begin{align}
		\binom{\lambda}{\mu}_\tau \coloneqq h_\mu(\overline\lambda;\tau).
	\end{align}
\end{defn}
The normalization here is called \mydef{unital} in the sense that $\binom{\mu}{\mu}=1$.
The \mydef{monic} normalization, denoted by $h_\mu^\monic$, is the one such that the coefficient of $m_\mu$ (or $P_\mu$) is 1. 

It is well-known that Schur polynomials and Jack polynomials can be given by tableau-sum formulas.
Okounkov found a similar formula for monic interpolation Jack polynomials, which we now recall. (We are following the notations in \cite{Koo15,CS24}.)
\begin{prop}\label{prop:tab}
	Schur polynomials, Jack polynomials and interpolation Jack polynomials can be given by the following tableau-sum formulas:
	\begin{alignat}{3}
		\mathrm{S}:	&\phantom{n}&&s_\lambda(x) &\quad=&\quad\sum_T \prod_{s\in\lambda} x_{T(s)},	\label{eqn:J-comb}\\
		\mathrm{J}:	&\phantom{n}&&P_\lambda(x;\tau) &\quad=&\quad\sum_T \psi_T(\tau)\prod_{s\in\lambda} x_{T(s)},	\label{eqn:J-comb}\\
		\AJ:	&&&	h_\lambda^\monic(x;\tau)	&=&\quad	\sum_T \psi_T(\tau)\prod_{s\in\lambda}\(x_{T(s)}-\(a_\lambda'(s)+(n-T(s)-l_\lambda'(s))\tau\)\),	\label{eqn:AJ-comb}
	\end{alignat}
	where $\mathrm{S}$ and $\mathrm{J}$ stand for Schur and Jack polynomials and $\AJ$ stands for (type $A$) interpolation Jack polynomials.
	Each sum runs over RTs of shape $\lambda$ and rank $n$. 
	The weight $\psi_T$, the statistics $a_\lambda'(s)$ and $l_\lambda'(s)$ can be found in \cite[Section 2]{CS24}.
\end{prop}
\begin{example}
	For example, for $\mu=(3,2)$ and $n=2$, we have
	\begin{align*}
		s_\mu=P_\mu=m_\mu=x_1^3x_2^2+x_1^2x_2^3,
	\end{align*}
	and
	\begin{align*}
		h_{\mu}^{\monic}
		&= 	x_{2}x_{1}(x_{2}-1)(x_{1}-1)(x_{1}-2-\tau )
		+x_{2}x_{1}(x_{2}-1)(x_{1}-1)(x_{2}-2) 	
		\\&=	x_{1}x_{2}(x_{1}-1)(x_{2}-1)(x_{1}+x_{2}-\tau-4).
	\end{align*}
	The definition involves the following 12 partitions
	\begin{align*}
		(0, 0), (1, 0), (1, 1), (2, 0), (2, 1), (3, 0), (2, 2), (3, 1), (4, 0), (3, 2), (4, 1), (5, 0).
	\end{align*}
	One can easily verify that $h_\mu^\monic$ indeed vanishes at $\overline\lambda=(\lambda_1+\tau,\lambda_2)$ for all but $\overline{(3,2)}$.
	One also sees that $h_{\mu}^{\monic}$ vanishes at $\overline{(m,0)}$ and $\overline{(m-1,1)}$ for $m\geq6$, more than required in the definition.
\end{example}
This phenomenon is true in general, called the \mydef{extra vanishing property}, proved in \cite{KS96}.
\begin{prop}\label{prop:extra}
	We have $\binom{\lambda}{\mu}_\tau = 0$ unless $\lambda\supseteq\mu$.
\end{prop}

\subsection{Positivity}
In this subsection, we recall the positivity cone $\fp$ for Jack polynomials \cite{Sahi-Jack,CS24}. 
The base field $\mathbb F$ is $\mathbb Q(\tau)$.
Let 
\begin{align}\label{eqn:AJ-F}
	\fp\coloneqq \Set*{\frac fg}{f,g\in\mathbb Z_{\geqslant0}[\tau],\ g\neq0}, \quad \fpp\coloneqq\fp\setminus\{0\}.
\end{align}
Then $\fp$ is a convex multiplicative cone, i.e., it is closed under addition, multiplication, and scalar multiplication by $\mathbb Q_{\geqslant0}$.
When we view $\tau$ as a real number instead of an indeterminate, we have $f(\tau)\geqslant0$ if $\tau\geq0$ for $f\in\fp$.
For example, by \cite[(VI.10.20)]{Mac15}, we have (our $\tau$ is equal to $1/\alpha$)
\begin{align}
	P_\lambda(\bm1;\tau) = \prod_{(i,j)\in\lambda} \frac{n+\alpha(j-1)-(i-1)}{\alpha(\lambda_i-j)+\lambda'_j-i+1} = \prod_{(i,j)\in\lambda} \frac{ (n-i+1)\tau+j-1}{(\lambda'_j-i+1)\tau+(\lambda_i-j)}\in\fpp. 
\end{align}
As mentioned in \cref{sec:intro}, it is proved in \cite{Sahi-Jack} that the binomial coefficients $\binom{\lambda}{\mu}_\tau$ is positive. More precisely, we have the following.
\begin{prop}[Positivity]\label{prop:pos}
	The binomial coefficient $\binom{\lambda}{\mu}_\tau\in\fpp$ if and only if $\lambda\supseteq\mu$.
\end{prop}
As a corollary of the binomial formula \cref{eqn:AJ-binomial1} and the positivity, we have:
\begin{cor}
	$P_\lambda(x+\bm1;\tau)$ is Jack positive over $\fp$.	
	In particular, 
	$s_\lambda(x+\bm1)$ is Schur positive, $m_\lambda(x+\bm1)$ is monomial positive, and $e_{\lambda'}(x+\bm1)$ is elementary positive.
\end{cor}

\section{Proof of \cref{thm:2}}\label{sec:mono}
After the positivity of the binomial coefficients, one asks for the monotonicity, i.e., \cref{thm:2}.
In this section, we give a sketch of its proof. See \cite{CS24} for full details.

We may assume that $\lambda$ is obtained from $\mu$ by adding one box in row $i_0$.
The normalizing factor $H_\nu(\tau) \coloneqq h_\nu^\monic(\overline\nu;\tau)$ and the weight $\psi_T(\tau)$ can be easily seen to be in $\fpp$.
Now apply the tableau-sum formula (\cref{prop:tab}) to write $h_\nu^\monic(\overline\lambda;\tau)-h_\nu^\monic(\overline\mu;\tau)$ as a sum over RTs.
For each $T$, the factors for $\lambda$ and $\mu$ are almost identical, apart from the ones with $T(s)=i_0$. 
For such $s\in\nu$, one can show that the difference is positive. 

\section{Symmetric function inequalities}\label{sec:ineq}
Using the binomial formula \cref{eqn:AJ-binomial1} and the monotonicity (\cref{thm:2}), we can find many interesting inequalities for symmetric functions.
\begin{thm}\label{thm:CS-ineq}
	Fix $\tau_0\in[0,\infty]$.
	The following statements are equivalent:
	\begin{enumerate}
		\item $\lambda\supseteq\mu$.
		\item $\displaystyle \frac{P_\lambda(x+\bm1;\tau)}{P_\lambda(\bm1;\tau)}-\frac{P_\mu(x+\bm1;\tau)}{P_\mu(\bm1;\tau)}$ is \textbf{Jack positive} over $\fp$.
		\item $\displaystyle \frac{P_\lambda(x+\bm1;\tau_0)}{P_\lambda(\bm1;\tau_0)}-\frac{P_\mu(x+\bm1;\tau_0)}{P_\mu(\bm1;\tau_0)}$, is $\tau_0$-\textbf{Jack positive} over $\R_{\geq0}$ (or over $\Q_{\geq0}$ if $\tau_0\in\Q$).
		
		As special cases of (3), 
		$\displaystyle \frac{s_\lambda(x+\bm1)}{s_\lambda(\bm1)}-\frac{s_\mu(x+\bm1)}{s_\mu(\bm1)}$ is \textbf{Schur positive},
		$\displaystyle \frac{m_\lambda(x+\bm1)}{m_\lambda(\bm1)}-\frac{m_\mu(x+\bm1)}{m_\mu(\bm1)}$ is \textbf{monomial positive},
		$\displaystyle \frac{Z_\lambda(x+\bm1)}{Z_\lambda(\bm1)}-\frac{Z_\mu(x+\bm1)}{Z_\mu(\bm1)}$ is \textbf{Zonal positive}, and
		$\displaystyle \frac{e_{\lambda'}(x+\bm1)}{e_{\lambda'}(\bm1)}-\frac{e_{\mu'}(x+\bm1)}{e_{\mu'}(\bm1)}$ is \textbf{elementary positive} (when expressed in $\Set{e_{\nu'}}{\nu\in\mathcal P_n}$).
		\item $\displaystyle \frac{p_{\lambda}(x+\bm1)}{p_{\lambda}(\bm1)}-\frac{p_{\mu}(x+\bm1)}{p_{\mu}(\bm1)}$ is \textbf{power-sum positive} (when expressed in $\Set{p_{\nu'}}{\nu\in\mathcal P_n}$).
	\end{enumerate}
\end{thm}
\begin{proof}
	Note that $P_\lambda(\bm1;\tau)\in\fpp$ and $P_\lambda(\bm1;\tau_0)>0$.
	
	We first show that (1)$\implies$(2).
	If $\lambda\supseteq\mu$, then by the binomial formula \cref{eqn:AJ-binomial1},
	\begin{align*}
		\frac{P_\lambda(x+\bm1;\tau)}{P_\lambda(\bm1;\tau)} -\frac{P_\mu(x+\bm1;\tau)}{P_\mu(\bm1;\tau)} = \sum_{\nu} \(\binom{\lambda}{\nu}_\tau -\binom{\mu}{\nu}_\tau\) \frac{P_\nu(x;\tau)}{P_\nu(\bm1;\tau)}.
	\end{align*}
	The coefficient $\binom{\lambda}{\nu}_\tau -\binom{\mu}{\nu}_\tau$ is in $\fp$ by \cref{thm:2}.
	
	(2)$\implies$(3) is clear since functions in $\fp$ have non-negative evaluation at $\tau_0\in[0,\infty]$.
	
	(3)$\implies$(1): Assume that $\lambda\not\supset\mu$, then the difference $\frac{P_\lambda(x+\bm1;\tau)}{P_\lambda(\bm1;\tau)} -\frac{P_\mu(x+\bm1;\tau)}{P_\mu(\bm1;\tau)}$ would contain the term $-\frac{P_\mu(x;\tau_0)}{P_\mu(\bm1;\tau_0)}$, contradicting (3).
	
	(1)$\iff$(4) is some easy computation, see \cite[Theorem 6.4]{CS24}.
\end{proof}

These inequalities are related to the ones studied by Muirhead, Cuttler--Greene--Skandera, Sra and Khare--Tao \cite{Muirhead,CGS11,Sra16,KT21}, which we now recall.
\begin{prop}\label{prop:CGS}
	Suppose $|\lambda|=|\mu|$. The following are equivalent:
		\begin{multicols}{2}
		\begin{enumerate}
			\item $\lambda$ dominates $\mu$.
			\item  $\dfrac{m_\lambda(x)}{m_\lambda(\bm1)}-\dfrac{m_\mu(x)}{m_\mu(\bm1)}\geq0$ for $x\in[0,\infty)^n$.
			\item  $\dfrac{e_{\lambda'}(x)}{e_{\lambda'}(\bm1)}-\dfrac{e_{\mu'}(x)}{e_{\mu'}(\bm1)}\geq0$ for $x\in[0,\infty)^n$.
			\item  $\dfrac{p_\lambda(x)}{p_\lambda(\bm1)}-\dfrac{p_\mu(x)}{p_\mu(\bm1)}\geq0$ for $x\in[0,\infty)^n$.
			\item  $\dfrac{s_\lambda(x)}{s_\lambda(\bm1)}-\dfrac{s_\mu(x)}{s_\mu(\bm1)}\geq0$ for $x\in[0,\infty)^n$.
			\item[\vspace{\fill}]
		\end{enumerate}			
		\end{multicols}
\end{prop}
\begin{prop}\label{prop:KT}
	The following are equivalent:
	\begin{multicols}{2}
	\begin{enumerate}
		\item $\lambda$ weakly dominates $\mu$.
		\item $\dfrac{s_\lambda(x)}{s_\lambda(\bm1)}-\dfrac{s_\mu(x)}{s_\mu(\bm1)}\geq0$ for $x\in[1,\infty)^n$.
	\end{enumerate}
	\end{multicols}
\end{prop}
Other such inequalities in \cref{prop:KT} follow from the previous two types \cref{thm:CS-ineq,prop:CGS}.
\begin{cor}\label{cor:KT}
	The following are equivalent:
	\begin{multicols}{2}
	\begin{enumerate}
		\item $\lambda$ weakly dominates $\mu$.
		\item $\dfrac{m_\lambda(x)}{m_\lambda(\bm1)}-\dfrac{m_\mu(x)}{m_\mu(\bm1)}\geq0$ for $x\in[1,\infty)^n$.
		\item $\dfrac{e_{\lambda'}(x)}{e_{\lambda'}(\bm1)}-\dfrac{e_{\mu'}(x)}{e_{\mu'}(\bm1)}\geq0$ for $x\in[1,\infty)^n$.
		\item $\dfrac{p_\lambda(x)}{p_\lambda(\bm1)}-\dfrac{p_\mu(x)}{p_\mu(\bm1)}\geq0$ for $x\in[1,\infty)^n$.
		\item $\dfrac{s_\lambda(x)}{s_\lambda(\bm1)}-\dfrac{s_\mu(x)}{s_\mu(\bm1)}\geq0$ for $x\in[1,\infty)^n$.
		\item[\vspace{\fill}]
	\end{enumerate}
	\end{multicols}
\end{cor}
\begin{proof}
	For briefness, We only show (1)$\implies$(5). There exists $\nu$ such that $\lambda\supseteq\nu$ and $\nu\m\mu$.
	\begin{align*}
		\dfrac{s_\lambda(x+\bm1)}{s_\lambda(\bm1)}-\dfrac{s_\mu(x+\bm1)}{s_\mu(\bm1)}
		=\left(\dfrac{s_\lambda(x+\bm1)}{s_\lambda(\bm1)}-\dfrac{s_\nu(x+\bm1)}{s_\nu(\bm1)}\right)+\left(\dfrac{s_\nu(x+\bm1)}{s_\nu(\bm1)}-\dfrac{s_\mu(x+\bm1)}{s_\mu(\bm1)}\right).
	\end{align*}
	Both differences are positive for $x\in[0,\infty)^n$ by \cref{thm:CS-ineq,prop:KT}.
	Each of the implications (1)$\implies$(2)--(4) is similar. Each of the reverse implications (2)--(5)$\implies$(1) is by some degree consideration.
\end{proof}
Moreover, our inequalities work for Jack polynomials as well.
It is natural to extend the inequalities of the second and third types to Jack polynomials:
\begin{conj}\label{conj:CGS-Jack}
	Fix $\tau_0 \in [0,\infty]$, and suppose $|\lambda|=|\mu|$. Then the following statements are
	equivalent:
	\begin{enumerate}
		\item $\lambda$ dominates $\mu$.
		\item $\displaystyle\frac{P_\lambda(x;\tau)}{P_\lambda(\bm1;\tau)} -
		\frac{P_\mu(x;\tau)}{P_\mu(\bm1;\tau)}\in\fp^{\mathbb R}=\{f/g\mid f(\tau)\in\R_{\geq0}[\tau],g\in\Z_{\geq0}[\tau]\setminus 0\}$, for $x\in[0,\infty)^n$.
		\item $\displaystyle\frac{P_\lambda(x;\tau_0)}{P_\lambda(\bm1;\tau_0)} -
		\frac{P_\mu(x;\tau_0)}{P_\mu(\bm1;\tau_0)}\geqslant 0$, for $x\in[0,\infty)^n$.
	\end{enumerate}
\end{conj}
\begin{conj}\label{conj:KT-Jack}
	Fix $\tau_0 \in [0,\infty]$. The following statements are equivalent:
	\begin{enumerate}
		\item $\lambda$ weakly dominates $\mu$.
		\item $\displaystyle\frac{P_\lambda(x;\tau)}{P_\lambda(\bm1;\tau)} -
		\frac{P_\mu(x;\tau)}{P_\mu(\bm1;\tau)}\in\fp^{\mathbb R}$, for $x\in[1,\infty)^n$.
		\item $\displaystyle\frac{P_\lambda(x;\tau_0)}{P_\lambda(\bm1;\tau_0)} -
		\frac{P_\mu(x;\tau_0)}{P_\mu(\bm1;\tau_0)}\geqslant 0$, for $x\in[1,\infty)^n$.
	\end{enumerate}
\end{conj}
Note that (2)$\implies$(3)$\implies$(1) in both conjectures are clear; and the implication (1)$\implies$(2) of \cref{conj:CGS-Jack} implies that of \cref{conj:KT-Jack} via the argument of \cref{cor:KT}. 

\begin{example}
	To illustrate Sra's inequality and the first conjecture, let $n=2$,  $\lambda=(4,0)$ and $\mu=(3,1)$, then we have 
	\begin{align*}
		\dfrac{s_\lambda(x)}{s_\lambda(\bm1)}-\dfrac{s_\mu(x)}{s_\mu(\bm1)}
		&=\frac{1}{15}(x_1-x_2)^2(3x_1^2+4x_1x_2+3x_2^2),	\\
		\frac{P_\lambda(x;\tau)}{P_\lambda(\bm1;\tau)} -
		\frac{P_\mu(x;\tau)}{P_\mu(\bm1;\tau)}
		&=\frac{(\tau+3)(x_1-x_2)^2}{4(2\tau+1)(2\tau+3)} \(\tau(x_1+x_2)^2+2(x_1^2+x_1x_2+x_2^2)\).
	\end{align*}
	Clearly both are positive when evaluating at $[0,\infty)^n$.
\end{example}

The authors, together with Apoorva Khare, are currently working on these conjectures and some further generalizations in \cite{CKS}.

\acknowledgements{We thank Apoorva Khare for many helpful discussions on the last section. The first author would like to thank Swee Hong Chan for his suggestions on the manuscript.}
\printbibliography

@misc{CS24,
	title={Interpolation Polynomials, Binomial Coefficients, and Symmetric Function Inequalities}, 
	author={Chen, Hong and Sahi, Siddhartha},
	year={2024},
	eprint={2403.02490},
	archivePrefix={arXiv},
	primaryClass={math.CO},
}

@misc{CKS,
	note={In preparation}, 
	author={Chen, Hong and Khare, Apoorva and Sahi, Siddhartha},
}

@article {Kad97,
	AUTHOR = {Kadell, Kevin W. J.},
	TITLE = {The {S}elberg-{J}ack symmetric functions},
	JOURNAL = {Adv. Math.},
	FJOURNAL = {Advances in Mathematics},
	VOLUME = {130},
	YEAR = {1997},
	NUMBER = {1},
	PAGES = {33--102},
	ISSN = {0001-8708,1090-2082},
	MRCLASS = {05E05 (33C45 33C80)},
	MRNUMBER = {1467311},
	MRREVIEWER = {Andrew\ Mathas},
	DOI = {10.1006/aima.1997.1642},
	URL = {https://doi.org/10.1006/aima.1997.1642},
}

@article {St89,
	AUTHOR = {Stanley, Richard P.},
	TITLE = {Some combinatorial properties of {J}ack symmetric functions},
	JOURNAL = {Adv. Math.},
	FJOURNAL = {Advances in Mathematics},
	VOLUME = {77},
	YEAR = {1989},
	NUMBER = {1},
	PAGES = {76--115},
	ISSN = {0001-8708,1090-2082},
	MRCLASS = {05A15 (20C30)},
	MRNUMBER = {1014073},
	MRREVIEWER = {Dennis\ White},
	DOI = {10.1016/0001-8708(89)90015-7},
	URL = {https://doi.org/10.1016/0001-8708(89)90015-7},
}

@book {Sagan91,
	AUTHOR = {Sagan, Bruce E.},
	TITLE = {The symmetric group},
	PUBLISHER = {Wadsworth \& Brooks/Cole Advanced Books \& Software, Pacific
	Grove, CA},
	YEAR = {1991},
	PAGES = {xviii+197},
	MRCLASS = {05E10 (05-01 05E05 20-01 20C30)},
	MRNUMBER = {1093239},
}

@misc{DD23,
title={Positive formula for Jack polynomials, Jack characters and proof of Lassalle's conjecture}, 
author={Ben Dali, Houcine and Do{\l}{\k e}ga, Maciej},
year={2023},
eprint={2305.07966},
archivePrefix={arXiv},
primaryClass={math.CO},
}

@article{KSinv,
AUTHOR = {Knop, Friedrich and Sahi, Siddhartha},
TITLE = {A recursion and a combinatorial formula for {J}ack
polynomials},
JOURNAL = {Invent. Math.},
FJOURNAL = {Inventiones Mathematicae},
VOLUME = {128},
YEAR = {1997},
NUMBER = {1},
PAGES = {9--22},
ISSN = {0020-9910,1432-1297},
MRCLASS = {33D80 (05E05 39A10)},
MRNUMBER = {1437493},
MRREVIEWER = {Kenji\ Taniguchi},
DOI = {10.1007/s002220050134},
URL = {https://doi.org/10.1007/s002220050134},
}

@incollection {Oko05,
	AUTHOR = {Okounkov, Andrei},
	TITLE = {The uses of random partitions},
	BOOKTITLE = {X{IV}th {I}nternational {C}ongress on {M}athematical
	{P}hysics},
	PAGES = {379--403},
	PUBLISHER = {World Sci. Publ., Hackensack, NJ},
	YEAR = {2005},
	ISBN = {981-256-201-X},
	MRCLASS = {05A17 (14N35 53D20 81T45 82-02 82B31 82B41)},
	MRNUMBER = {2227852},
	MRREVIEWER = {Richard\ Kenyon},
}

@article {LV96,
	AUTHOR = {Lapointe, Luc and Vinet, Luc},
	TITLE = {Exact operator solution of the {C}alogero-{S}utherland model},
	JOURNAL = {Comm. Math. Phys.},
	FJOURNAL = {Communications in Mathematical Physics},
	VOLUME = {178},
	YEAR = {1996},
	NUMBER = {2},
	PAGES = {425--452},
	ISSN = {0010-3616,1432-0916},
	MRCLASS = {81V70 (33D80 82B23)},
	MRNUMBER = {1389912},
	MRREVIEWER = {Kazuhiro\ Hikami},
	URL = {http://projecteuclid.org/euclid.cmp/1104286659},
}

@book {Stanley,
	AUTHOR = {Stanley, Richard P.},
	TITLE = {Enumerative combinatorics. {V}ol. 2},
	SERIES = {Cambridge Studies in Advanced Mathematics},
	VOLUME = {62},
	PUBLISHER = {Cambridge University Press, Cambridge},
	YEAR = {1999},
	PAGES = {xii+581},
	MRCLASS = {05A15 (05-02 05E05 05E10 68R05)},
	MRNUMBER = {1676282},
	MRREVIEWER = {Ira\ Gessel},
	DOI = {10.1017/CBO9780511609589},
	URL = {https://doi.org/10.1017/CBO9780511609589},
}

@article {KT99,
	AUTHOR = {Knutson, Allen and Tao, Terence},
	TITLE = {The honeycomb model of {${\rm GL}_n({\bf C})$} tensor
	products. {I}. {P}roof of the saturation conjecture},
	JOURNAL = {J. Amer. Math. Soc.},
	FJOURNAL = {Journal of the American Mathematical Society},
	VOLUME = {12},
	YEAR = {1999},
	NUMBER = {4},
	PAGES = {1055--1090},
	ISSN = {0894-0347,1088-6834},
	MRCLASS = {20G05 (05E05 14M15 52B20)},
	MRNUMBER = {1671451},
	MRREVIEWER = {Anders\ Skovsted\ Buch},
	DOI = {10.1090/S0894-0347-99-00299-4},
	URL = {https://doi.org/10.1090/S0894-0347-99-00299-4},
}

@article {BZ92,
	AUTHOR = {Berenstein, A. D. and Zelevinsky, A. V.},
	TITLE = {Triple multiplicities for {${\rm sl}(r+1)$} and the spectrum
	of the exterior algebra of the adjoint representation},
	JOURNAL = {J. Algebraic Combin.},
	FJOURNAL = {Journal of Algebraic Combinatorics. An International Journal},
	VOLUME = {1},
	YEAR = {1992},
	NUMBER = {1},
	PAGES = {7--22},
	ISSN = {0925-9899,1572-9192},
	MRCLASS = {17B10},
	MRNUMBER = {1162639},
	MRREVIEWER = {Dao\ Ji\ Meng},
	DOI = {10.1023/A:1022429213282},
	URL = {https://doi.org/10.1023/A:1022429213282},
}

@incollection {Sahi94,
	AUTHOR = {Sahi, Siddhartha},
	TITLE = {The spectrum of certain invariant differential operators
	associated to a {H}ermitian symmetric space},
	BOOKTITLE = {Lie theory and geometry},
	SERIES = {Progr. Math.},
	VOLUME = {123},
	PAGES = {569--576},
	PUBLISHER = {Birkh\"{a}user Boston, Boston, MA},
	YEAR = {1994},
	ISBN = {0-8176-3761-3},
	MRCLASS = {43A85 (22E30 22E46)},
	MRNUMBER = {1327549},
	MRREVIEWER = {Ji\ Sheng\ Na},
	DOI = {10.1007/978-1-4612-0261-5\_21},
	URL = {https://doi.org/10.1007/978-1-4612-0261-5_21},
}

@article {Lascoux,
	AUTHOR = {Lascoux, Alain},
	TITLE = {Classes de {C}hern d'un produit tensoriel},
	JOURNAL = {C. R. Acad. Sci. Paris S\'er. A-B},
	FJOURNAL = {Comptes Rendus Hebdomadaires des S\'eances de l'Acad\'emie des
	Sciences. S\'eries A et B},
	VOLUME = {286},
	YEAR = {1978},
	NUMBER = {8},
	PAGES = {A385--A387},
	MRCLASS = {14F05 (14M15)},
	MRNUMBER = {476745},
}

@book {Muirhead,
    AUTHOR = {Muirhead, R. F.},
     TITLE = {Some {M}ethods {A}pplicable to {I}dentities and {I}nequalities
              of {S}ymmetric {A}lgebraic {F}unctions of n {L}etters},
      NOTE = {Thesis (D.Sc.)--University of Glasgow (United Kingdom)},
 PUBLISHER = {ProQuest LLC, Ann Arbor, MI},
      YEAR = {1904},
     PAGES = {34},
   MRCLASS = {99-05},
  MRNUMBER = {4051404},
}

@article {Jack,
    AUTHOR = {Jack, Henry},
     TITLE = {A class of symmetric polynomials with a parameter},
   JOURNAL = {Proc. Roy. Soc. Edinburgh Sect. A},
  FJOURNAL = {Proceedings of the Royal Society of Edinburgh. Section A.
              Mathematics},
    VOLUME = {69},
      YEAR = {1970/71},
     PAGES = {1--18},
   MRCLASS = {12.30 (10.00)},
  MRNUMBER = {289462},
}

@book {Muirhead82,
    AUTHOR = {Muirhead, Robb J.},
     TITLE = {Aspects of multivariate statistical theory},
    SERIES = {Wiley Series in Probability and Mathematical Statistics},
 PUBLISHER = {John Wiley \& Sons, Inc., New York},
      YEAR = {1982},
     PAGES = {xix+673},
   MRCLASS = {62Hxx},
  MRNUMBER = {652932},
MRREVIEWER = {Nariaki\ Sugiura},
}

@article {Sahi98,
    AUTHOR = {Sahi, Siddhartha},
     TITLE = {The binomial formula for nonsymmetric {M}acdonald polynomials},
   JOURNAL = {Duke Math. J.},
  FJOURNAL = {Duke Mathematical Journal},
    VOLUME = {94},
      YEAR = {1998},
    NUMBER = {3},
     PAGES = {465--477},
   MRCLASS = {33C50 (33C80)},
  MRNUMBER = {1639523},
MRREVIEWER = {V.\ L.\ Deshpande},
       DOI = {10.1215/S0012-7094-98-09419-4},
       URL = {https://doi.org/10.1215/S0012-7094-98-09419-4},
}

@article {Las90,
	AUTHOR = {Lassalle, Michel},
	TITLE = {Une formule du bin\^{o}me g\'{e}n\'{e}ralis\'{e}e pour les
	polyn\^{o}mes de {J}ack},
	JOURNAL = {C. R. Acad. Sci. Paris S\'{e}r. I Math.},
	FJOURNAL = {Comptes Rendus de l'Acad\'{e}mie des Sciences. S\'{e}rie I.
	Math\'{e}matique},
	VOLUME = {310},
	YEAR = {1990},
	NUMBER = {5},
	PAGES = {253--256},
	MRCLASS = {05E05 (05A10 33C50 33C80)},
	MRNUMBER = {1042857},
	MRREVIEWER = {Laurent\ Habsieger},
}

@article {Kan93,
	AUTHOR = {Kaneko, Jyoichi},
	TITLE = {Selberg integrals and hypergeometric functions associated with
	{J}ack polynomials},
	JOURNAL = {SIAM J. Math. Anal.},
	FJOURNAL = {SIAM Journal on Mathematical Analysis},
	VOLUME = {24},
	YEAR = {1993},
	NUMBER = {4},
	PAGES = {1086--1110},
	MRCLASS = {33C80},
	MRNUMBER = {1226865},
	MRREVIEWER = {Kevin\ W. J. Kadell},
	DOI = {10.1137/0524064},
	URL = {https://doi.org/10.1137/0524064},
}

@article{OO97,
	AUTHOR = {Okounkov, A. and Olshanski, G.},
	TITLE = {Shifted {J}ack polynomials, binomial formula, and
	applications},
	JOURNAL = {Math. Res. Lett.},
	FJOURNAL = {Mathematical Research Letters},
	VOLUME = {4},
	YEAR = {1997},
	NUMBER = {1},
	PAGES = {69--78},
	MRCLASS = {05E05 (33D80)},
	MRNUMBER = {1432811},
	MRREVIEWER = {Peter\ J.\ Forrester},
	DOI = {10.4310/MRL.1997.v4.n1.a7},
	URL = {https://doi.org/10.4310/MRL.1997.v4.n1.a7},
}

@article{Oko98-Mac,
	AUTHOR = {Okounkov, Andrei},
	TITLE = {({S}hifted) {M}acdonald polynomials: {$q$}-integral
	representation and combinatorial formula},
	JOURNAL = {Compositio Math.},
	FJOURNAL = {Compositio Mathematica},
	VOLUME = {112},
	YEAR = {1998},
	NUMBER = {2},
	PAGES = {147--182},
	MRCLASS = {05E05},
	MRNUMBER = {1626029},
	MRREVIEWER = {Ang\`ele\ M.\ Hamel},
	DOI = {10.1023/A:1000436921311},
	URL = {https://doi.org/10.1023/A:1000436921311},
}

@article{KS96,
	AUTHOR = {Knop, Friedrich and Sahi, Siddhartha},
	TITLE = {Difference equations and symmetric polynomials defined by
	their zeros},
	JOURNAL = {Internat. Math. Res. Notices},
	FJOURNAL = {International Mathematics Research Notices},
	YEAR = {1996},
	NUMBER = {10},
	PAGES = {473--486},
	MRCLASS = {05E05 (39A10)},
	MRNUMBER = {1399412},
	DOI = {10.1155/S1073792896000311},
	URL = {https://doi.org/10.1155/S1073792896000311},
}

@article {OO-schur,
	AUTHOR = {Okunkov, A. and Olshanski, G.},
	TITLE = {Shifted {S}chur functions},
	JOURNAL = {Algebra i Analiz},
	FJOURNAL = {Rossi\u{\i}skaya Akademiya Nauk. Algebra i Analiz},
	VOLUME = {9},
	YEAR = {1997},
	NUMBER = {2},
	PAGES = {73--146},
	MRCLASS = {05E05 (05E10 17B35 20G05)},
	MRNUMBER = {1468548},
	MRREVIEWER = {Witold\ Kra\'{s}kiewicz},
}

@article {Sahi-Jack,
	AUTHOR = {Sahi, Siddhartha},
	TITLE = {Binomial coefficients and {L}ittlewood-{R}ichardson
	coefficients for {J}ack polynomials},
	JOURNAL = {Int. Math. Res. Not. IMRN},
	FJOURNAL = {International Mathematics Research Notices. IMRN},
	YEAR = {2011},
	NUMBER = {7},
	PAGES = {1597--1612},
	MRCLASS = {05E05 (06A06 33C52)},
	MRNUMBER = {2806516},
	MRREVIEWER = {Irina\ Sviridova},
	DOI = {10.1093/imrn/rnq126},
	URL = {https://doi.org/10.1093/imrn/rnq126},
}

@article{CGS11,
	AUTHOR = {Cuttler, Allison and Greene, Curtis and Skandera, Mark},
	TITLE = {Inequalities for symmetric means},
	JOURNAL = {European J. Combin.},
	FJOURNAL = {European Journal of Combinatorics},
	VOLUME = {32},
	YEAR = {2011},
	NUMBER = {6},
	PAGES = {745--761},
	MRCLASS = {05E05},
	MRNUMBER = {2821548},
	MRREVIEWER = {Trueman\ MacHenry},
	DOI = {10.1016/j.ejc.2011.01.020},
	URL = {https://doi.org/10.1016/j.ejc.2011.01.020},
}

@book {Mac15,
	AUTHOR = {Macdonald, I. G.},
	TITLE = {Symmetric functions and {H}all polynomials},
	SERIES = {Oxford Classic Texts in the Physical Sciences},
	EDITION = {Second},
	PUBLISHER = {The Clarendon Press, Oxford University Press, New York},
	YEAR = {2015},
	MRCLASS = {05E05 (01A75 05-02 20C30 20C33 20K01 33C80 33D80)},
	MRNUMBER = {3443860},
}

@article{Koo15,
	AUTHOR = {Koornwinder, Tom H.},
	TITLE = {Okounkov's {$BC$}-type interpolation {M}ac{D}onald polynomials
	and their {$q=1$} limit},
	JOURNAL = {S\'{e}m. Lothar. Combin.},
	FJOURNAL = {S\'{e}minaire Lotharingien de Combinatoire},
	VOLUME = {72},
	YEAR = {2014/15},
	PAGES = {Art. B72a, 27},
	MRCLASS = {33D52},
	MRNUMBER = {3383151},
	MRREVIEWER = {Marzena\ Szajewska},
}

@article{Sra16,
	AUTHOR = {Sra, Suvrit},
	TITLE = {On inequalities for normalized {S}chur functions},
	JOURNAL = {European J. Combin.},
	FJOURNAL = {European Journal of Combinatorics},
	VOLUME = {51},
	YEAR = {2016},
	PAGES = {492--494},
	MRCLASS = {05E05},
	MRNUMBER = {3398873},
	DOI = {10.1016/j.ejc.2015.07.005},
	URL = {https://doi.org/10.1016/j.ejc.2015.07.005},
}

@article {KT21,
	AUTHOR = {Khare, Apoorva and Tao, Terence},
	TITLE = {On the sign patterns of entrywise positivity preservers in
	fixed dimension},
	JOURNAL = {Amer. J. Math.},
	FJOURNAL = {American Journal of Mathematics},
	VOLUME = {143},
	YEAR = {2021},
	NUMBER = {6},
	PAGES = {1863--1929},
	MRCLASS = {15B48 (05E05 15B35 26A48 30B10)},
	MRNUMBER = {4349135},
	MRREVIEWER = {Raffael\ Hagger},
	DOI = {10.1353/ajm.2021.0049},
	URL = {https://doi.org/10.1353/ajm.2021.0049},
}
\end{document}